\documentclass[12pt, fms,
  times,
]{cuparticle}

\usepackage{amsmath,amssymb}

\font\cyr=wncyr8

\font\itcyr=wncyi8

\DeclareMathOperator{\GL}{GL}

\newcommand{\e}{\varepsilon}
\newcommand{\IB}{{\mathbb B}}
\newcommand{\IK}{{\mathbb K}}
\newcommand{\IM}{{\mathbb M}}
\newcommand{\IN}{{\mathbb N}}
\newcommand{\IZ}{{\mathbb Z}}
\newcommand{\IR}{{\mathbb R}}
\newcommand{\IC}{{\mathbb C}}
\newcommand{\cA}{{\mathcal A}}

\newcommand{\cC}{{\mathcal C}}

\newcommand{\cM}{{\mathcal M}}
\newcommand{\cQ}{{\mathcal Q}}
\newcommand{\cU}{{\mathcal U}}

\newcommand{\cCN}{{\mathcal C}{(\IN)}}

\newcommand{\G}{\Gamma}

\DeclareMathOperator{\cspan}{\overline{span}}

\DeclareMathOperator{\dist}{dist}
\DeclareMathOperator{\Ad}{Ad}
\DeclareMathOperator{\bc}{H_b^1}

\newcommand{\tn}[1]{\mathopen{|\hspace*{-1pt}|\hspace*{-1pt}|}#1\mathclose{|\hspace*{-1pt}|\hspace*{-1pt}|}}

\newtheorem{thm}{Theorem}
\newtheorem{lem}[thm]{Lemma}
\newtheorem{prop}[thm]{Proposition}

\newtheorem{defn}[thm]{Definition}
\newproof{proof}{Proof}

\volume{1}
\doi{10.111}

\begin{document}

\authorheadline{Y. Choi, I. Farah and N. Ozawa}
\runningtitle{An amenable operator algebra}

\begin{frontmatter}

\title{A nonseparable amenable operator algebra
which is not isomorphic to a $\mathrm{C}^*$-alge\-bra}
\author{Yemon Choi}

\ead{choi@math.usask.ca}

\author{Ilijas Farah}
\ead{ifarah@mathstat.yorku.ca}

\author{Narutaka Ozawa}
\cormark[1]
\cortext[1]{Corresponding author}
\ead{narutaka@kurims.kyoto-u.ac.jp}

\address{Department of Mathematics and Statistics \\
University of Saskatchewan \\
McLean Hall, 106 Wiggins Road \\
Saskatoon, Saskatchewan, Canada S7N 5E6}
\address{Department of Mathematics and Statistics, York University, 4700 Keele Street, North York, Ontario, Canada, M3J 1P3, and Matematicki Institut, Kneza Mihaila 35, Belgrade, Serbia}

\address{RIMS, Kyoto University, \mbox{606-8502} Japan}
\makeatletter
\let\@wraptoccontribs\wraptoccontribs
\makeatother


\date{\today}

\begin{abstract}
It has been a longstanding problem whether every amenable operator algebra is isomorphic
to a (necessarily nuclear) $\mathrm{C}^*$-alge\-bra. In this note, we give a nonseparable
counterexample. The existence of a separable counterexample remains an open problem.
We also initiate a general study of unitarizability of representations of amenable
groups in $\mathrm{C}^*$-algebras
and show that our method cannot produce a separable counterexample.
\end{abstract}
\end{frontmatter}
\section{Introduction}
The notion of amenability for Banach algebras was introduced by B.~E.~Johnson (\cite{johnson})
in 1970s and has been studied intensively since then (see a more recent monograph \cite{runde}).
For several natural classes of Banach algebras, the amenability
property is known to single out the ``good" members of those classes. For example, B.~E.~Johnson's
fundamental observation (\cite{johnson}) is that the Banach algebra $L^1(G)$ of
a locally compact group $G$ is amenable if and only if the group $G$ is amenable.
Another example is the celebrated result of Connes (\cite{connes}) and Haagerup (\cite{haagerup})
which states that a $\mathrm{C}^*$-alge\-bra is amenable as a Banach algebra if and only if it is nuclear.

In this paper, we are interested in the class of \emph{operator algebras}.
By an operator algebra, we mean a (not necessarily self-adjoint)
norm-closed subalgebra of $\IB(H)$, the $\mathrm{C}^*$-alge\-bra of the bounded linear operators
on a Hilbert space~$H$. It has been asked by several researchers whether every amenable
operator algebra is isomorphic to a (necessarily nuclear) $\mathrm{C}^*$-alge\-bra.
The problem has been solved affirmatively in several special cases:
for subalgebras
of commutative $\mathrm{C}^*$-alge\-bras (\cite{sheinberg}),
and subsequently for operator algebras generated by normal elements (\cite{cl});
for subalgebras of compact operators (\cite{willis,gifford});
for $1$-amenable operator algebras (Theorem 7.4.18 in~\cite{bl});
and for commutative subalgebras of finite von Neumann algebras (\cite{choi}).

Here we give the first counterexample to the above problem.
In fact, our counterexample is a subalgebra of the homogeneous
$\mathrm{C}^*$-alge\-bra $\ell_\infty(\IN,\IM_2)$.
Hence the result of \cite{sheinberg} is actually quite sharp and
the result of \cite{choi} does not generalize to
an arbitrary subalgebra of a finite von Neumann algebra.

\begin{thm}\label{thm:main}
There is a unital amenable operator algebra $\cA$
which is not isomorphic to a $\mathrm{C}^*$-alge\-bra.
The algebra $\cA$ is a subalgebra of $\ell_\infty(\IN,\IM_2)$
with density character~$\aleph_1$, and
is an inductive limit of unital separable subalgebras $\{ \cA_i\}_{i<\aleph_1}$,
each of which is conjugated to a $\mathrm{C}^*$-sub\-alge\-bra of $\ell_\infty(\IN,\IM_2)$
by an invertible element $v_i\in\ell_\infty(\IN,\IM_2)$, such that
$\sup_i \|v_i\| \|v_i^{-1}\| <\infty$.
Moreover, for any $\e>0$, one can choose $\cA$ to 
be $(1+\e)$-amenable.
\end{thm}

Here, $C$-amenable means that the amenability constant is at most $C$ (see Definition 2.3.15 in \cite{runde}).
One drawback of our counterexample is that it is inevitably nonseparable,
 as explained by Theorem~\ref{thm:cd1s} below,
and the existence of a separable counterexample remains an open problem.
We note that if such an example exists,
 then there is one
among subalgebras of the finite von Neumann algebra
$\prod_{n=1}^\infty\IM_n$. Indeed, by Voiculescu's theorem~(\cite{voiculescu}),
 the cone $C_0((0,1],\cA)$
of a separable operator algebra $\cA$ can be realized as a closed subalgebra of
$\prod_{n=1}^\infty\IM_n/\bigoplus_{n=1}^\infty\IM_n$.
The cone of $\cA$ is amenable (see Exercise 2.3.6 in \cite{runde}),
and its preimage $\tilde{\cA}$ in $\prod_{n=1}^\infty\IM_n$ is
 an extension of the cone by the amenable algebra $\bigoplus_{n=1}^\infty\IM_n$,
 hence $\tilde{\cA}$ is amenable (see Theorem 2.3.10~in \cite{runde}).
$\tilde{\cA}$ is not isomorphic to a $\mathrm{C}^*$-algebra,
since it has $\cA$ as a quotient and every closed two-sided ideal in
a $\mathrm{C}^*$-algebra is automatically $*$-closed.

\subsection*{Acknowledgment.}
This joint work was initiated when the second and third authors participated in
the workshop ``$\mathrm{C}^*$-Alge\-bren" (ID:1335) held at the
Mathe\-mati\-sches Forschungs\-insti\-tut Ober\-wolfach in August 2013.
We are grateful to the organizers S.~Echterhoff, M.~R{\o}rdam, S.~Vaes, and D. Voiculescu,
and the institute for giving the authors an opportunity of a joint work.
We are also grateful to N.~C.~Phillips for useful conversations during the workshop and
helpful remarks on the first version of this paper,
and the third author would like to thank N.~Monod and H.~Matui for
valuable conversations.
Finally, we would like to thank S. A. White for his encouragement to include
the last sentence of
Theorem~\ref{thm:main}.

The first author was supported by NSERC Discovery Grant 402153-2011.
The second author was partially supported by NSERC, a Velux Visiting Professorship, and the Danish Council for Independent Research through Asger T\"ornquist's grant, no. 10-082689/FNU.
The third author was partially supported by JSPS (23540233).

\section{Proof of Theorem~\ref{thm:main}}\label{sec:pfmain}
Let $\cC$ be a unital $\mathrm{C}^*$-alge\-bra, $\G$ be a group,
and $\pi\colon\G\to \cC$ be a representation,
i.e., $\pi(s)$ is invertible for every $s\in\G$
and $\pi(st)=\pi(s)\pi(t)$ for every $s,t\in\G$.
The representation $\pi$ is said to be \emph{uniformly bounded} if
$\|\pi\|:=\sup_s\|\pi(s)\|<+\infty$.
It is said to be \emph{unitarizable} if
there is an invertible element $v$ in $\cC$ such that $\Ad_v\circ\pi$
is a unitary representation. Here $\Ad_v(c)=vcv^{-1}$ for $c\in\cC$.
The element $v$ is called a \emph{similarity} element.
A well-known theorem of Sz.-Nagy, Day, Dixmier, and Nakamura--Takeda
states that every uniformly bounded representation
of an amenable group $\G$ into a von Neumann algebra is unitarizable.
In fact the latter property characterizes amenability by
Pisier's theorem (\cite{pisier07}).
In particular, the operator algebra $\cspan\pi(\G)$ generated by
a uniformly bounded representation $\pi$ of an amenable group $\G$
is an amenable operator algebra which is isomorphic to a nuclear $\mathrm{C}^*$-alge\-bra.
See \cite{pisier} and \cite{runde} for general information about uniformly bounded representations
and amenable Banach algebras, respectively.

Let us fix the notation. Let $\IM_2$ be the $2$-by-$2$ full matrix algebra,
$\ell_\infty(\IN,\IM_2)$ be the $\mathrm{C}^*$-alge\-bra of the bounded sequences in $\IM_2$,
and $c_0(\IN,\IM_2)$ be the ideal of the sequences that converge to zero.
We shall freely identify $\ell_\infty(\IN,\IM_2)$ with
$\ell_\infty(\IN)\otimes\IM_2$,
and $\ell_\infty(\IN,\IM_2)/c_0(\IN,\IM_2)$ with $\cC(\IN)\otimes\IM_2$,
where $\cC(\IN)=\ell_\infty(\IN)/c_0(\IN)$.
The quotient map from $\ell_\infty(\IN)$ (or $\ell_\infty(\IN)\otimes\IM_2$)
onto $\cC(\IN)$ (or $\cC(\IN)\otimes\IM_2$) is denoted by $Q$.

\begin{lem}\label{lem:ext}
Let $\G$ be an abelian group and $\pi\colon\G\to\cC(\IN)\otimes\IM_2$
be a uniformly bounded representation. Then the amenable operator algebra
\[
\cA:=Q^{-1}(\cspan\pi(\G))\subset\ell_\infty(\IN,\IM_2)
\]
is isomorphic to a $\mathrm{C}^*$-alge\-bra if and only if $\pi$ is unitarizable.
\end{lem}
\begin{proof}
First of all, we observe that the operator algebra $\cA$ is indeed amenable because
it is an extension of an amenable Banach algebra $\cspan\pi(\G)$ by the
amenable Banach algebra $c_0(\IN,\IM_2)$ (see Theorem~2.3.10 in \cite{runde}).
Suppose now that $\pi$ is unitarizable and $v\in\cC(\IN)\otimes\IM_2$ has the property
$\Ad_v\circ\pi$ is unitary.
We may assume $v$ is positive, by taking the positive component from its polar decomposition.
Since $v$ is invertible, we can choose
 a representing sequence $v_m$, for $m\in \IN$ of $v$ such that each $v_m$ is positive and
moreover $1/\|v^{-1}\|\leq v_m \leq \|v\|$ for all $m$. In particular each $v_m$ is invertible and
$\|v_m\|\|v_m^{-1}\|\leq \|v\| \|v^{-1}\|$
for all $m$.
Now we have a representing sequence of an  invertible lift $\tilde{v}\in\ell_\infty(\IN,\IM_2)$
of $v$ such that $\|\tilde{v}\| \|\tilde{v}^{-1}\|=\| v\| \|v^{-1}\|$.
Then $\tilde{v}\cA\tilde{v}^{-1}=Q^{-1}(\cspan(\Ad_v\circ\pi(\G)))$
is a self-adjoint $\mathrm{C}^*$-sub\-alge\-bra of $\ell_\infty(\IN,\IM_2)$. Conversely, suppose that
$\cA$ is isomorphic to a $\mathrm{C}^*$-alge\-bra, which is necessarily nuclear.
Then thanks to the solution of Kadison's similarity problem for
nuclear $\mathrm{C}^*$-alge\-bras (see Theorem~7.16 in \cite{pisier} or Theorem~1 in \cite{pisier07}),
there is $\tilde{v}$ in the von Neumann algebra $\ell_\infty(\IN,\IM_2)$
such that $\tilde{v}\cA\tilde{v}^{-1}$ is a $\mathrm{C}^*$-sub\-alge\-bra.
Let $v=Q(\tilde{v})\in\cC(\IN)\otimes\IM_2$.
Since $Q(\tilde{v}\cA\tilde{v}^{-1})$ is a commutative $\mathrm{C}^*$-sub\-alge\-bra of $\cC(\IN)\otimes\IM_2$,
for every $s \in \G$, the element $v\pi(s)v^{-1}$ is normal with its spectrum in the unit circle,
which implies that $v\pi(s)v^{-1}$ is unitary.
\end{proof}

The above proof uses the fact that   every (not necessarily separable)
 amenable C*-algebra is nuclear, as well as  the
solution to Kadison's similarity problem for nuclear C*-algebras.
The reader may appreciate  a more elementary and
self-contained proof.
 Assume $\theta$ is a bounded homomorphism of a unital C*-algebra $\cA$ into $\ell_\infty(\IN,\IM_2)$.
 We need to prove that $\theta$ is similar to a *-homomorphism.
It  suffices to show that every coordinate map is similar to a *-homomorphism
and that the similarities are implemented by a uniformly bounded sequence $v_n$, for $n\in \IN$,
of operators.
Consider the restriction of $\theta$ to the unitary group $G$ of~$\cA$.
At the $n$-th coordinate we have a bounded homomorphism from $G$ to $\GL(2,\IC)$.
Since a bounded subgroup of $\GL(2,\IC)$ is included in a compact subgroup,
by a standard averaging argument we  find~$v_n$ such that $\Ad_{v_n}\circ\theta$ is a unitary representation of $G$. The operators $v_n$ are easily seen to satisfy the required properties.


\begin{proof}[of Theorem~$\ref{thm:main}$]
We consider two $2$-by-$2$ order $2$ invertible matrices which are not simultaneously unitarizable.
For instance, let
$s^0=\left[\begin{smallmatrix}1&0\\0&-1\end{smallmatrix}\right]$
and
$s^1=\left[\begin{smallmatrix}1&0\\1&-1\end{smallmatrix}\right]$.
Then by compactness, one has
\[
\e(C):=\inf\{ d(v s^0 v^{-1},\cU)+d(v s^1 v^{-1},\cU) : v\in\IM_2^{-1},\ \|v\|\|v^{-1}\|\le C\} > 0
\]
for every $C>0$. Here $\cU$ denotes the unitary group of $\IM_2$.

We shall need two families $\{E_i^0: i\in \aleph_1\}$ and $\{E_i^1: i\in \aleph_1\}$
of subsets of $\IN$ such that (i) $E_i^k\cap E_j^l$ is finite whenever $(i,k)\neq (j,l)$
and (ii) these two families are not \emph{separated},
in the sense that there is no $F\subseteq \IN$ such that both $E_i^0\setminus F$
and $E_i^1\cap F$ are finite for all $i$.
The existence of such pair of families follows from
  \cite{Lu47}. Luzin actually proved much more: he constructed a single family
    $\{E_i: i<\aleph_1\}$
of infinite subsets of $\IN$ such that (i)
 $E_i\cap E_j$ is finite whenever $i\neq j$ and (ii) whenever $X\subseteq \aleph_1$
 is such that both $X$ and $\aleph_1\setminus X$ are uncountable, then
 the families $\{E_i: i\in X\}$ and $\{E_i: i\in \aleph_1\setminus X\}$ cannot be separated
 (see 
 \ref{sec:luzin} below for Luzin's proof).

 The projections $p_i^k=Q(1_{E_i^k})\in\cCN$ are mutually orthogonal.
For each pair $(i,k)$, we define $s_i^k$  in $\cCN\otimes\IM_2$ by
\[
s_i^k = p_i^k \otimes s^k + (1-p_i^k) \otimes 1.
\]
Let $\G:=\bigoplus_{i\in \aleph_1, k\in \{0,1\}}\IZ/2\IZ$ and $\{e_i^k\}$ be its standard basis.
Then the map
$e_i^k\mapsto s_i^k$
extends to a uniformly bounded representation $\pi\colon \G\to\cCN\otimes\IM_2$ such that
$\|\pi\|=\max\{\|s^0\|,\|s^1\|\}$.
We claim that $\pi$ is not unitarizable.
Suppose for a contradiction that there is
an invertible element $v\in\cCN\otimes\IM_2$ such that
$\Ad_{v}\circ\pi$ is unitary. As in the proof of Lemma~\ref{lem:ext}
we may assume $v$ is positive
and find a representing sequence $v_m$, for $m\in \IN$,
of an invertible lift of $v$ such that  $\|v_m\|\|v_m^{-1}\|\leq \|v\|\|v^{-1}\|$ for all $m$.
Let $\e=\e(\|v\|\|v^{-1}\|)$.

Now let $F^0:=\{m: d(v_ms^0v_m^{-1},\cU)<\e/2\}$, and note that this
set is disjoint from $F^1:=\{m: d(v_m s^1 v_m^{-1},\cU)<\e/2\}$.
Therefore we have $i$ such that $E_i^0\setminus F^0$ is infinite or
such that $E_i^1\setminus F^1$ is infinite.
If the former case applies, then
\[
\limsup_{n\in E_i^0,\, n\to\infty} d(v_ns^0v_n^{-1},\cU) \geq  \e/2,
\]
contradicting the assumption that $v$ unitarizes $\pi$.
The  case when $E_i^1\setminus F_1$ is infinite similarly leads to a contradiction.
Thus, by Lemma~\ref{lem:ext}, the preimage of $\cspan\pi(\G)$
in $\ell_\infty(\IN,\IM_2)$ is an amenable operator algebra which is not isomorphic
to a $\mathrm{C}^*$-alge\-bra. Its density character is equal to $\aleph_1=|\G|$.


Let $\Gamma_i$ be a countable subgroup of $\Gamma$ and denote the separable algebra $Q^{-1}(\cspan\pi(\Gamma_i))$ by $\cA_i$.
Theorem~\ref{thm:cd1s} below shows that $\cA_i$ is similar inside $\ell_\infty(\IN,\IM_2)$ to an amenable $\mathrm{C}^*$-algebra,
with a similarity element $v_i$ satisfying $\|v_i\| \|v_i^{-1}\| \leq \|\pi\|^2$.
Furthermore, since every amenable $\mathrm{C}^*$-algebra is $1$-amenable by results of Haagerup~(\cite{haagerup}),
 $\cA_i$ is $\|\pi\|^4$-amenable.
Now $\cA$ is the inductive limit of the family $(\cA_i)$ as $\Gamma_i$ varies over all countable subgroups of $\Gamma$. Since each $\cA_i$ is $\|\pi\|^4$-amenable, a routine argument with approximate diagonals shows that $\cA$ is also $\|\pi\|^4$-amenable: for details see Proposition 2.3.17 in~\cite{runde}.

Finally, we explain how our example can be modified to have arbitrarily small amenability constant.
For $0<t<1$, we keep $s^0= \left[\begin{smallmatrix}1&0\\0&-1\end{smallmatrix}\right]$ but replace $s^1$
with $s^1(t) = \left[\begin{smallmatrix}1&0\\t&-1\end{smallmatrix}\right]$ in our original construction.
Denoting the resulting algebra by~$\cA(t)$, the previous arguments show that $\cA(t)$ is $\|s^1(t)\|^4$-amenable, and $\|s^1(t)\|$
can be made arbitarily close to $1$.
\end{proof}

We note that a set-theoretical study of the cohomological nature of gaps similar to Luzin's
was initiated in \cite{Talayco}.

\section{Unitarizability of uniformly bounded representations}
In this section, we develop a general study of (non-)unitarizability.
First, we shall deal with separable $\mathrm{C}^*$-alge\-bras.
Let $\cA$ be a unital $\mathrm{C}^*$-alge\-bra and $\theta$ be a $*$-auto\-mor\-phism on $\cA$.
An element $a\in\cA$ is called a \emph{cocycle} if it satisfies
\[
\tn{a}:=\sup_{n\geq1}\|\sum_{k=0}^{n-1}\theta^k(a)\|<+\infty.
\]
It is \emph{inner} (or a coboundary) if there is $x\in\cA$ such that $a=x-\theta(x)$.
We recall that the \emph{first bounded cohomology
group} (see \cite{monod2001continuous}) of the $\IZ$-module
$(\cA,\theta)$ is defined as
\[
\bc(\cA,\theta)=\{\mbox{ cocycles }\}/\{\mbox{ inner cocycles }\}.
\]
When $\cA$ is abelian and $\theta$ corresponds to a minimal homeomorphism
of its spectrum then  $\bc$ is trivial (see Theorem~2.6 in~\cite{ormes}).

We note that every cocycle is approximately inner.
Indeed, since $a_n:=\sum_{k=0}^{n-1}\theta^k(a)$ satisfies $a_{n+1}=a+\theta(a_n)$,
the element $x_n:= n^{-1}\sum_{m=1}^n a_m$ satisfies $\|x_n\|\le\tn{a}$ and
$\| a-(x_n-\theta(x_n)) \| \le 2n^{-1}\tn{a}$.
Suppose for a moment that $\theta$ is inner, $\theta=\Ad_u$ for a unitary element $u\in\cA$,
and $a\in\cA$ is a cocycle.
Then, $t=\left[\begin{smallmatrix} u & au\\ 0 & u\end{smallmatrix}\right]$ is an invertible
element in $\IM_2(\cA)$ such that
$t^n=\left[\begin{smallmatrix} u^n & a_nu^n\\ 0 & u^n\end{smallmatrix}\right]$
for $n\geq1$.
Therefore $\sup_{n\in\IZ}\|t^n\|\le1+\tn{a}$ and
$t$ gives rise to a uniformly bounded representation $\pi_a$ of $\IZ$ into $\cA$.
\begin{lem}\label{lem:der}
Let $\cA$, $u$, $a$, and $\pi_a$ as above.
Then the uniformly bounded representation~$\pi_a$ is
unitarizable if and only if $a$ is inner.
\end{lem}
See Lemma~4.5 in \cite{pisier} or \cite{mo} for the proof of this lemma.

\begin{prop}\label{prop:nonzerobc}
Let $\cA$ be a unital separable $\mathrm{C}^*$-alge\-bra and $\theta$ be a $*$-automorphism of $\cA$.
Suppose that there are a (non-unital) $\theta$-invariant $\mathrm{C}^*$-sub\-alge\-bra $\cA_0$,
a state~$\phi$ on $\cA_0$, and a sequence of natural numbers $n(k)$ such that
$(\phi\circ\theta^{n(k)})_{k=1}^\infty$ converges to~$0$ pointwise on $\cA_0$.
Then, $\bc(\cA,\theta)\neq0$.
\end{prop}

\begin{proof}
By a standard Hahn--Banach convexity argument, we construct
 an approximate unit $(h_n)_{n=0}^\infty$ of $\cA_0$ such that
$0\leq h_n\leq 1$, $h_{n+1}h_n=h_n$, and $\|h_n-\theta(h_n)\|<2^{-n}$ for all $n$.
We note that $\phi'(h_n)\to1$ for any state $\phi'$ on $\cA_0$.
Taking a state extension, we may assume that $\phi$ is defined on $\cA$.
Since $\cA$ is separable, passing to a subsequence, we may assume that
$\phi^k:=\phi\circ\theta^{n(k)}$ converges pointwise to a state, say $\psi$, on $\cA$.

Set $k(1)=1$. By induction, one can find strictly increasing
sequences $(m(j))_{j=1}^\infty$ and $(k(j))_{j=1}^\infty$ of natural numbers
such that
$\phi^{k(i)}(h_{m(j)})>1-2^{-j}$ for every $i\le j$ and
$\phi^{k(j+1)}(h_{m(j)})<2^{-j}$ for every $j$.
Let
\[
x=\mbox{SOT-}\sum_{j=1}^\infty(h_{m(2j)}-h_{m(2j-1)})\in \cA^{**}.
\]
We extend $\theta$ and $\phi$ on $\cA^{**}$ by ultraweak continuity.
One has $a:=x-\theta(x)\in \cA$, since it is a norm-convergent series in $\cA_0$.
By a telescoping argument,  $a$ is a cocycle.

Suppose for the sake of obtaining a contradiction that $a$ is inner
and $x-\theta(x)=y-\theta(y)$ for some $y\in\cA$.
Then, $y\in\cA$ and $\theta(x-y)=x-y$. It follows that
$\phi^{k(j)}(y)\to\psi(y)$ and $\phi^{k(j)}(x-y)=\phi(x-y)$.
Hence the sequence $( \phi^{k(j)}(x))_{j=1}^\infty$ converges.
However, for $j\geq1$,
\[
\phi^{k(2j)}(x) \geq \phi^{k(2j)}(h_{m(2j)} - h_{m(2j-1)})
\geq 1-\frac{1}{2^{2j}}-\frac{1}{2^{2j-1}}
\]
and
\begin{align*}
\phi^{k(2j+1)}(x)
\le \phi^{k(2j+1)}(\sum_{i=1}^{j} h_{m(2i)})
+\sum_{i=j+1}^{\infty}(1-\phi^{k(2j+1)}( h_{m(2i-1)}))
\le\frac{1}{4}.
\end{align*}
Hence, the sequence $( \phi^{k(j)}(x))_{j=1}^\infty$ does not converge, and we have a contradiction.
\end{proof}

An example of  $\cA_0$, $\phi$ and $\theta$ as in the statement of Proposition~\ref{prop:nonzerobc}
are the ideal~$\IK$   of compact operators on $\IB(\ell_2(\IZ))$, any one of its states, and the bilateral
shift on $\ell_2(\IZ)$.

\begin{lem}\label{lem:unbc}
For every unital separable $\mathrm{C}^*$-alge\-bra $\cA$ which is not of type $\mathrm{I}$,
there is a unitary element $u\in\cA$ such that $\bc(\cA,\Ad_u)\neq0$.
\end{lem}

\begin{proof}
Let $z$ be the bilateral shift on $\ell_2(\IZ)$ and take a selfadjoint element $h\in\IB(\ell_2(\IZ))$
such that $z=\exp(\sqrt{-1}h)$. Let $\cC\subset\IB(\ell_2(\IZ))$ be the unital $\mathrm{C}^*$-sub\-alge\-bra
generated  by $\IK$ and $h$, and let $\phi_0$ be the vector state at $\delta_0$.
Since $\cC$ is an extension of a commutative $\mathrm{C}^*$-alge\-bra by $\IK$, it is nuclear.
By Kirchberg's theorem and Glimm's theorem in tandem (Corollary 1.4(vii) in \cite{kirchberg}),
there are a unital $\mathrm{C}^*$-sub\-alge\-bra
$\cA_1$ of $\cA$ and a surjective $*$-homo\-mor\-phism $\pi$ from $\cA_1$ onto $\cC$.
Let $g\in\cA_1$ be a selfadjoint lift of $h$ and let $u:=\exp(\sqrt{-1}g)\in\cA_1$,
which is a unitary lift of $z$.
Then $\cA_0=\pi^{-1}(\IK)$ is an $\Ad_u$-invariant subalgebra
and the state $\phi=\phi_0\circ\pi$ satisfies $\phi\circ(\Ad_u)^n\to0$ pointwise on $\cA_0$.
Hence the result follows from Proposition~\ref{prop:nonzerobc}.
\end{proof}

Combining Lemma~\ref{lem:unbc} and Lemma~\ref{lem:der}, we arrive at the following theorem.

\begin{thm}
For every unital separable $\mathrm{C}^*$-alge\-bra $\cA$ which is not of type $\mathrm{I}$,
there is a uniformly bounded representation of $\IZ$ into $\IM_2(\cA)$ which is not unitarizable.
\end{thm}

Now, we shall deal with non-separable $\mathrm{C}^*$-alge\-bras.
Our approach uses model theory of metric structures
and the extension of Pedersen's techniques \cite{Pede:Corona} as presented in \cite{fh}.
The following is Definition~1.1 from \cite{fh}, with a misleading typo corrected.

\begin{defn}
Given a $\mathrm{C}^*$-alge\-bra $\cM$, a \emph{degree-$1$ *-polynomial}
with coefficients in~$\cM$
is a linear combination of terms of the form $ax b$, $a x^* b$ and $a$ with $a,b$ in $\cM$.
A $\mathrm{C}^*$-alge\-bra $\cM$ is said to be \emph{countably degree-$1$ saturated}
if for every countable family of degree-$1$ $*$-poly\-no\-mials $P_n(\bar{x})$
with coefficients in $\cM$ and variables $x_m$, for $m\in\IN$, and every family of compact sets
$K_n\subset\IR$, for $n\in\IN$, the following are equivalent
(writing $\bar{b}$ for $(b_1,b_2,\ldots)$ and $\cM_{\le1}$ for the closed unit ball of $\cM$).
\begin{enumerate}
\item There are $b_m\in\cM_{\le1}$, for $m\in\IN$, such that $\|P_n(\bar{b})\|\in K_n$ for all $n$.
\item For every $N\in\IN$ there are $b_m\in\cM_{\le1}$, for $m\in\IN$, such that
\[
\dist(\|P_n(\bar{b})\|, K_n)\le\frac{1}{N}
\]
for all $n\le N$.
\end{enumerate}
\end{defn}

A type $\{P_n(\bar x)\in K_n: n\in \IN\}$ satisfying (1) is said to be \emph{realized} in $\cM$ and a type satisfying (2) is said to be
\emph{consistent}
with (or \emph{approximately finitely realized in})  $\cM$.
Coronas of $\sigma$-unital $\mathrm{C}^*$-algebras, in particular
the Calkin algebra $\cQ(\ell_2)$ and $\cC(\IN)\otimes\IM_2$,  as well
as ultraproducts associated with nonprincipal ultrafilters on~$\IN$,
are countably degree-$1$ saturated
(\cite[Theorem~1.4]{fh}).
In each of these cases, given a consistent type, a realization $\bar b$ is assembled from the approximate realizations
$\bar b^n$, for $n\in \IN$ and a carefully chosen, appropriately
quasicentral approximate unit $e_n$, for $n\in \IN$, as $\bar b=\sum_n (e_n-e_{n+1})^{1/2} \bar b^n (e_n-e_{n+1})^{1/2}$.
See \cite{fh} for details and
more examples of countably degree-$1$ saturated $\mathrm{C}^*$-alge\-bras.

\begin{thm}\label{thm:cd1s}
Let $\cM$ be a unital countably degree-$1$ saturated $\mathrm{C}^*$-alge\-bra.
Then, every uniformly bounded representation $\pi\colon\G\to\cM$ of
a countable amenable group~$\G$ into $\cM$ is unitarizable. Moreover a similarity
element $v$ can be chosen so that it satisfies $\|v\| \|v^{-1}\|\le\|\pi\|^2$.
\end{thm}

\begin{proof}
The proof is analogous to the standard one (see Theorem 0.6 in \cite{pisier}),
modulo applying countable degree-$1$ saturation.
Consider  the type in variable $x$
over $\cM$ consisting of conditions $\|x-x^*\|=0$, $\|x\|\leq \|\pi\|^2$, $\|\hspace{.5pt}\|\pi\|^2-x\| \le\|\pi\|^2-\|\pi\|^{-2}$,
and $\| \pi(s) x \pi(s)^* - x \|=0$ for all $s\in \G$.

We now check that this type is consistent.
Let $(F_n)_{n=1}^\infty$ be a F{\o}lner sequence of finite subsets of $\G$.
Then,
\[
h_n=\frac{1}{|F_n|}\sum_{t\in F_n}\pi(t)\pi(t)^*,
\]
are positive elements in $\cM$ such that $\|\pi\|^{-2}\le h_n \le\|\pi\|^2$
and
\[
\| \pi(s) h_n \pi(s)^* - h_n \|\le \frac{|F_n\bigtriangleup sF_n|}{|F_n|}\|\pi\|^2 \to 0
\]
for every $s\in\G$. Hence this type is consistent and  by countable degree-$1$ saturation there is
$h\in\cM$ which realizes it. Therefore we have
$h=h^*$, $\|h\|\le\|\pi\|^2$, $\|\hspace{.5pt}\|\pi\|^2-h \|\le \|\pi\|^2-\|\pi\|^{-2}$, and
$\pi(s) h \pi(s)^* = h $ for every $s\in\G$.
It follows that $h$ is a positive element such that $\|\pi\|^{-2}\le h \le\|\pi\|^2$
and the invertible elements $h^{-1/2}\pi(s)h^{1/2}$ satisfy
\[
(h^{-1/2}\pi(s)h^{1/2})(h^{-1/2}\pi(s)h^{1/2})^*=h^{-1/2}\pi(s)h\pi(s)^*h^{-1/2}=1,
\]
i.e., $h^{-1/2}\pi(s)h^{1/2}$ are unitary.
\end{proof}

Theorem~\ref{thm:cd1s} shows that the method used in the proof of
Theorem~\ref{thm:main} cannot be used to produce a separable counterexample.

\appendix
\section{A correction for \cite{choi}}
We take the opportunity to fill a small gap in \cite{choi}.
The main result of that paper is only proved for commutative amenable subalgebras of
$\sigma$-finite finite von Neumann algebras. It is then stated in \cite{choi} that the general
case follows from the $\sigma$-finite one because any finite von Neumann algebra $\cM$
decomposes as a direct product $\prod_i \cM_i$ where each $\cM_i$ is $\sigma$-finite.
However, the example of the present paper shows that similarity to a $\mathrm{C}^*$-algebra
is not preserved by taking inductive limits, even with a uniform bound on the similarity elements, 
so more justification is needed. Instead, we may
argue as follows. Let $\cA$ be an amenable subalgebra of $\cM$ and let $\cA_i$ be its
image under the projection $\cM\to\cM_i$. Applying the main result of \cite{choi} to
each $\cA_i$, we obtain a uniformly bounded family $v_i\in\cM_i$ such that
$v_i \cA_i v_i^{-1}$ is a commutative $\mathrm{C}^*$-sub\-alge\-bra of~$\cM_i$.
Take $v$ to be the direct product of the $v_i$. Then $v\cA v^{-1}$ is an amenable
subalgebra of the commutative $\mathrm{C}^*$-alge\-bra
  $\prod_i v_i \cA_i v_i^{-1}$, and hence by \cite{sheinberg}
it is self-adjoint.

\section{A construction of Luzin's gap}\label{sec:luzin}
For the reader's convenience we prove Luzin's theorem.
Following von Neumann, we identify $n\in \IN$ with the set $\{0,1,\dots, n-1\}$.
We construct a family $E_i$, for $i<\aleph_1$, of infinite subsets of $\IN$
such that
\begin{enumerate}
\item $E_i\cap E_j$ is finite whenever $i\neq j$, and
\item for every $i$ and every $m\in \IN$ the set
$\{j<i: E_j\cap E_i\subseteq m\}$ is finite.
\end{enumerate}
The construction is by recursion.
For a finite  $i$ let $E_i=\{2^i(2k+1):k\in \IN\}$.
Assume  $i<\aleph_1$ is infinite  and the sets $E_j$, for $j<i$ were chosen to satisfy the
requirements.
Since $i$ is countable, we can re-enumerate $E_j$, for $j<i$ as $F_n$, for $n\in \IN$.

Now let $k(0)=0$ and $k(n)=\min F_n\setminus (k(n-1)\cup \bigcup_{l<n} F_l)$ for $n\geq 1$.
The sequence $\{k(n)\}$ is strictly increasing and $k(n)\in F_l$ implies $n\leq l$.
 Therefore $E_i=\{k(n): n\in \IN\}$  is infinite
 and $E_i\cap F_n\subseteq \{k(0), \dots, k(n)\}$ is finite for all $n$.
 Finally, for any $m\in \IN$ the set $\{n\in \IN: F_n\cap E_i\subseteq m\}
 \subseteq \{n: k(n)<m\}$ is finite.

 This describes  the recursive construction of a family $E_i$, for $i<\aleph_1$, satisfying (1) and (2).

 We claim that for any $X\subseteq \aleph_1$ such that $X$ and $\aleph_1\setminus X$ are uncountable
 the families $\{E_i: i\in X\}$ and $\{E_i: i\in \aleph_1\setminus X\}$ cannot be separated.
  Assume otherwise, and fix $F\subseteq \IN$ separating them.
  Since $E_i\setminus F$ is finite for all $i\in X$, there is $m\in \IN$ such that
  $X'=\{i\in X: E_i\setminus F\subseteq m\}$ is uncountable.
  By increasing $m$ if
  necessary we can assure that $Y'=\{i\in \aleph_1\setminus X: E_i\cap F\subseteq m\}$
  is also uncountable.

  Pick $i\in Y'$ such that $X''=\{j\in X': j<i\}$ is infinite.
Then for each $j\in X''$ we have    $E_j\cap E_i\subseteq (E_j\setminus F)\cup (E_i\cap F)
\subseteq m$. But this contradicts (2).

\end{document}